\documentclass[12pt]{article}
\usepackage[latin9]{inputenc}
\usepackage{amsthm}
\usepackage{amsmath}
\usepackage{amssymb}
\usepackage{a4wide}

\makeatletter

\newtheorem{theorem}{Theorem}

\newtheorem{lemma}[theorem]{Lemma}

\newtheorem{corollary}[theorem]{Corollary}

\def \RR {\mathbb R}

\def \EE {\mathbb E}
\def \ZZ {\mathbb Z}

\def \TT {\mathbb T}
\def \eps {\varepsilon}

\usepackage[english]{babel}

\title{A Lipschitz Function on a Torus is Almost Constant on a Large Parallel Subsection}
\date{}

\makeatother

\usepackage{babel}

\newcommand\blfootnote[1]{%
  \begingroup
  \renewcommand\thefootnote{}\footnote{#1}%
  \addtocounter{footnote}{-1}%
  \endgroup
}

\begin{document}
\global\long\def\dist{\mbox{dist}}

\global\long\def\Osc{\mbox{Osc}}

\title{On the oscillation rigidity of a Lipschitz function on a high-dimensional
flat torus }

\author{Dmitry Faifman, Bo'az Klartag and Vitali Milman}

\maketitle

\abstract{Given an arbitrary $1$-Lipschitz function $f$ on the torus $\TT^n $, we find a $k$-dimensional subtorus $M \subseteq \TT^n$, parallel to the axes,
such that the restriction of $f$ to the subtorus $M$ is nearly a constant function. The $k$-dimensional subtorus $M$ is chosen randomly and uniformly. We show
that when $k \leq c \log n / (\log \log n + \log 1/\eps)$, the maximum and the minimum of $f$ on this  random subtorus $M$ differ
by at most $\eps$, with high probability.}

\blfootnote{School of Mathematical Sciences, Tel Aviv University, Tel Aviv 69978, Israel. \\ E-mails: {\tt dfaifman@gmail.com}, {\tt klartagb@tau.ac.il}, {\tt milman@tau.ac.il}}

\section{Introduction}

A uniformly continuous function $f$ on an $n$-dimensional space $X$ of finite volume tends to concentrate near a single value as $n$ approaches
infinity, in the
sense that the $\eps$-extension of some level set has nearly
full measure. This phenomenon, which is called the {\it concentration of measure in high dimension},
is frequently related to a transitive group
of symmetries acting on $X$.
The prototypical example is the case of a $1$-Lipschitz function
on the unit sphere $S^n$, see \cite{MS,Le, Gr2}.  \\

One of the most important consequences of the concentration of measure
is the emergence of {\it spectrum}, as was discovered in
the 1970-s by the third named author, see \cite{M1,M2,M3}.
The idea is that not only the distinguished level set has a large
$\eps$-extension in sense of measure, but actually one may find
structured subsets on which the function is nearly constant.
When we have a group $G$ acting transitively on $X$,
this structured subset  belongs to the orbit $\{ g M_0 \, ; g \in G \}$
where $M_0 \subset X$ is a fixed subspace. The third named author
noted also some connections with Ramsey theory, which were developed
in two different directions: by Gromov in \cite{Gr} in the direction
of metric geomery, and by Pestov \cite{P1,P2} in the unexpected direction
of dynamical systems. \\

The phenomenon of spectrum thus follows from concentration, and
it is no surprise that most of the results in Analysis establishing spectrum
appeared as a consequence of concentration.
 In this note, we demonstrate an instance where no concentration of measure is available, but
nevertheless a geometrically structured level set arises. \\

To state our result, consider the standard flat torus $\TT^n = \RR^n /  \ZZ^n = (\RR / \ZZ)^n$,
which inherits its Riemannian structure from $\RR^n$.
 We say that $M\subset\mathbb{T}^{n}$ is a
{\it coordinate subtorus of dimension $k$} if it is the collection of all $n$-tuples $(\theta_{j})_{j=1}^{n}\in\mathbb{T}^{n}$
with fixed $n-k$ coordinates. Given a manifold $X$ and $f: X \rightarrow \RR$
we denote the oscillation of $f$ along $X$ by
\[
\Osc(f;X)=\sup_{X}f-\inf_{X}f.
\]

\begin{theorem}
There is a universal constant $c>0$, such that for any $n \geq 1, 0 < \eps \leq 1$
and a function  $f:\mathbb{T}^{n}\to\mathbb{R}$
which is $1$-Lipschitz,  there exists a $k$-dimensional
coordinate subtorus $M \subset\mathbb{T}^{n}$ with $k = \left \lfloor c\frac{\log n}{\log\log (3n)+\log |\eps|} \right \rfloor$,
such that $\Osc(f;M) \leq \eps$. \label{thm_1219}
\end{theorem}

Note that the collection of all coordinate subtori equals the orbit  $\{ g M_0 \, ; \, g \in G \}$
where $M_0 \subset \TT^n$ is any fixed $k$-dimensional coordinate subtorus, and the group $G = \RR^n \rtimes S_n$
acts on $\TT^n$ by translations and permutations of the coordinates. Theorem \ref{thm_1219} is a manifestation
of {\it spectrum}, yet its proof below is inspired by proofs of the Morrey embedding theorem,
and the argument does not follow the usual concentration paradigm. We think that the spectrum phenomenon should
be much more widespread, perhaps even more than the concentration phenomenon, and we hope that this note will be a small step towards
its recognition. \\

\subsubsection*{Acknowledgements}

We would like to thank Vladimir Pestov for his interest in this work.
The second-named author was supported by a grant from the European Research Council (ERC).

\section{Proof of the theorem}
  We write $| \cdot |$ for the standard Euclidean norm in $\RR^n$
  and we write $\log$ for the natural logarithm.
The standard vector fields $\partial/\partial x_{1},\ldots,
\partial/\partial x_n$ on $\RR^n$ are well defined also on the quotient $\TT^n = \RR^n / \ZZ^n$.
These $n$ vector fields are the ``coordinate directions'' on the unit torus $\TT^n$. Thus,
the partial derivatives $\partial_1 f, \ldots, \partial_n f$ are well-defined for any smooth
function $f: \TT^n \rightarrow \RR$, and we have $|\nabla f|^2 = \sum_{i=1}^n (\partial_i f)^2$.
A $k$-dimensional subspace $E\subset T_{x}\mathbb{T}^{n}$ is a {\it coordinate subspace}
if it is spanned by $k$ coordinate directions.  For $f: \TT^n \rightarrow \RR$ and $M\subset\mathbb{T}^{n}$
a submanifold, we write $\nabla_{M}f$ for the gradient of the restriction
$f|_{M}:M\to\mathbb{R}$.  \\

Throughout the proof,
$c,C$ will always denote universal constants, not necessarily the
same at each appearance.
Since the Riemannian volume of $\TT^n$ equals one, Theorem \ref{thm_1219} follows from the
case $\alpha = 1$ of the following:

\begin{theorem} Let $n \geq 1, 0 < \eps \leq 1, 0 < \alpha \leq 1$ and $1 \leq k \leq c \frac{\log n}{\log\log (5n)+|\log \eps| + |\log \alpha|}$.
Let $f:\mathbb{T}^{n}\to\mathbb{R}$ be a locally-Lipschitz function such that, for $p = k(1 + \alpha)$,
\begin{equation}  \int_{\TT^n} |\nabla f|^p  \leq 1. \label{eq_1714} \end{equation}
Then there exists a $k$-dimensional
coordinate subtorus $M \subset\mathbb{T}^{n}$ with $\Osc(f;M) \leq \eps$. \label{thm_138}
\end{theorem}

The plan of the proof is as follows. First, for some large $k$ we
find a $k$-dimensional coordinate subtorus $M$ where the derivative
is small on average, in the sense that $\Big(\int_{M}|\nabla_{M}f|^{p}\Big)^{1/p}$
is small. The existence of such a subtorus is
a consequence of the observation that
at every point most of the partial derivatives in the coordinate directions
are small. We then restrict our attention to this subtorus, and take
any two points $\tilde{x},\tilde{y}\in M$. Our goal is to show
that $f(\tilde{x}) - f(\tilde{y}) < \eps$. \\

To this end we construct a polygonal line from $\tilde{x}$ to $\tilde{y}$
which consists of intervals of length $1/2$. For every such interval $[x,y]$
we randomly select a point $Z$ in a $(k-1)$-dimensional  ball which is orthogonal
to the interval $[x,y]$ and is centered at its midpoint. We then show that
$|f(x) - f(Z)|$ and $|f(y) - f(Z)|$ are typically small, since $|\nabla_M f|$ is
small on average along the intervals $[x,Z]$ and $[y,Z]$. \\

We proceed with a formal proof of Theorem \ref{thm_138}, beginning with the following computation:

\begin{lemma}
For any $n \geq 1, 0 < \eps  \leq 1, 0 < \alpha \leq 1$ and $1 \leq k \leq c \frac{\log n}{\log\log (5n)+|\log \eps| + |\log \alpha|}$,
we have that  $k \leq n/2$ and
\begin{equation}
\Bigg(\frac{2 k}{\delta^{2}n}\Bigg)^{1/p} \leq \sqrt{k} \cdot \delta
 \label{eq_1236} \end{equation}
 where $p = (1 + \alpha) k$ and
 \begin{equation}  \delta = \frac{\alpha}{16(1+\alpha)} \cdot \frac{\eps}{k^{3/2}}.
 \label{eq_1019} \end{equation}
\label{lem_1256}
\end{lemma}
\begin{proof} Take $c = 1/200$. The desired conclusion (\ref{eq_1236})  is equivalent to
$4 k^{2-p} \leq \delta^{2p+4} n^2$, which in turn is equivalent to
\begin{equation}
2^{8p + 18} \cdot \left( \frac{\alpha + 1}{\alpha} \right)^{2p+4} \cdot  k^{2p + 8} \leq \eps^{2p+4} n^2. \label{eq_1119} \end{equation}
Since $c\leq 1/12$ we have that $6 p \leq 12 k \leq \log n / |\log \eps|$ and hence
$\eps^{2p+4} n^2 \geq \eps^{6p} n^2 \geq n$. Since $\alpha + 1 \leq 2$ then in order to obtain (\ref{eq_1119}) it suffices to prove
\begin{equation}
 \left( \frac{32}{\alpha} \cdot k \right)^{2p+8} \leq n. \label{eq_1242}
\end{equation}
Since $c \leq 1/24$ and $k \leq c \log n / (\log \log (5n))$ then
$24 k \log k \leq \log n$. Since
$k \leq c \frac{\log n}{|\log \alpha| + \log(\log 5)}$
then $24 k \log \left( \frac{32}{\alpha} \right) \leq \log n$.
We conclude that $ 12 k \log \left( \frac{32}{\alpha} \cdot k \right) \leq \log n$, and hence
\begin{equation}    \left( \frac{32}{\alpha} \cdot k \right)^{12 k} \leq n. \label{eq_1145} \end{equation}
However, $p = (1 + \alpha) k$ and hence $2p + 8 \leq 12 k$. Therefore the desired bound (\ref{eq_1242}) follows from (\ref{eq_1145}).
Since $k \leq \frac{1}{2} \log n \leq n/2$, the lemma is proven.
\end{proof}

Our standing assumptions for the remainder of the proof of Theorem \ref{thm_138} are that $n \geq 1, 0 < \eps \leq 1, 0 < \alpha \leq 1$ and
\begin{equation}  1 \leq k \leq c \frac{\log n}{\log\log (5n)+|\log \eps| + |\log \alpha|}  \label{eq_1300} \end{equation}
where $c > 0$ is the constant from Lemma \ref{lem_1256}. We also denote
\begin{equation} p = (1 + \alpha) k
\label{eq_2153} \end{equation}
 and
we write $e_1,\ldots,e_n$ for the standard $n$ unit vectors in $\RR^n$.

\begin{lemma} Let $v \in \RR^n$ and let $J \subset \{ 1, \ldots, n \}$ be a random subset of size $k$, chosen uniformly
from the collection of all $\left( \! \! \begin{array}{c} n \\ k \end{array} \! \! \right)$ subsets.
Consider the $k$-dimensional subspace $E \subset \RR^n$ spanned by $ \{ e_{j} ; j\in J\}$
and let $P_{E}$ be the orthogonal projection operator onto $E$ in $\RR^n$. Then,
\[
\big(\mathbb{E}|P_{E}v|^{p}\big)^{1/p}\leq
\frac{\alpha}{8(1+\alpha)} \cdot \frac{\eps}{k}
\cdot |v|.
\] \label{lem_subset}
\end{lemma}
\begin{proof}
We may assume that $v = (v_1,\ldots,v_n) \in \RR^n$ satisfies $|v| = 1$.
Let $\delta > 0$ be defined as in (\ref{eq_1019}).
Denote $I=\{i ; |v_{i}|\geq \delta \}$.
Since $|v|=1$, we must have $|I|\leq 1 / \delta^2$. We claim that
\begin{equation}
\mathbb{P}(I\cap J=\emptyset) \geq 1-\frac{2 k}{\delta^{2}n}. \label{eq_1042}
\end{equation}
Indeed, if $\frac{2 k}{\delta^{2}n} \geq 1$ then (\ref{eq_1042}) is obvious. Otherwise, $|I| \leq \delta^{-2} \leq n/2 \leq n-k$ and
\[
\mathbb{P}(I\cap J=\emptyset)=\prod_{j=0}^{k-1}\frac{n-|I|-j}{n-j}
\geq \left( 1 - \frac{|I|}{n-k+1} \right)^k
\geq \left( 1-\frac{2}{\delta^{2}n} \right)^k
\geq 1-\frac{2 k }{\delta^{2}n}.
\]
Thus (\ref{eq_1042}) is proven. Consequently,
\[
\mathbb{E}|P_{E}v|^{p} = \EE \left( \sum_{j\in J}v_{j}^2 \right)^{p/2}
\leq\frac{2 k}{\delta^{2}n}+\mathbb{E} \left[ 1_{\{ I \cap J = \emptyset \}} \cdot \left(  \sum_{j\in J}v_{j}^2 \right)^{p/2}  \right]
\leq\frac{2 k}{\delta^{2}n}+\Bigg(k \cdot \delta^2 \Bigg)^{p/2},
\]
where $1_A$ equals one if the event $A$ holds true and it vanishes otherwise.
By using the inequality $(a+b)^{1/p}\leq a^{1/p}+b^{1/p}$
we obtain
\[
\big(\mathbb{E}|P_{E}v|^{p}\big)^{1/p}\leq \Bigg(\frac{2 k}{\delta^{2}n}\Bigg)^{1/p}+\sqrt{k} \cdot \delta \leq
2  \sqrt{k} \cdot \delta  = \frac{\alpha}{8(1+\alpha)} \cdot \frac{\eps}{k},
\]
where we used (\ref{eq_1019}) and Lemma \ref{lem_1256}.
\end{proof}

\begin{corollary}
Let $f:\mathbb{T}^{n}\to\mathbb{R}$ be a locally-Lipschitz function with $ \int_{\TT^n} |\nabla f|^p  \leq 1$.
Then there exists a $k$-dimensional coordinate
subtorus $M \subset \TT^n$
such that \begin{equation} \left(\int_{M}|\nabla_{M}f|^{p}\right)^{1/p}\leq
\frac{\alpha}{8(1+\alpha)} \cdot \frac{\eps}{k}.
\label{eq_1309} \end{equation}
\label{lem:choiceofM}
\end{corollary}

\begin{proof}
The set of all coordinate $k$-dimensional subtori admits a unique probability
measure, invariant under translations and coordinate permutations.
Let $M$ be a random coordinate $k$-subtorus, chosen with respect to the uniform
distribution. All the tangent spaces $T_{x}\mathbb{T}^{n}$ are canonically
identified with $\mathbb{R}^{n}$, and we let $E\subset\mathbb{R}^{n}$
denote a random, uniformly chosen $k$-dimensional coordinate subspace.
Then we may write
\[
\mathbb{E}_{M}\int_{M}|\nabla_{M}f|^{p}=\int_{\mathbb{T}^{n}}\mathbb{E}_{E}|P_{E}\nabla f|^{p}\leq A^{p}\int_{\mathbb{T}^{n}}|\nabla f|^{p}\leq
A^{p},
\]
where $A = \frac{\alpha}{8(1+\alpha)} \cdot \frac{\eps}{k}$
and we used Lemma \ref{lem_subset}.
It follows that there exists $M$ that satisfies (\ref{eq_1309}).
\end{proof}

The following lemma is essentially  Morrey's inequality (see \cite[Section 4.5]{EG}).

\begin{lemma} Consider the $k$-dimensional Euclidean ball $B(0,R) = \{ x \in \RR^k \, ; \, |x| \leq R \}$.
Let $f: B(0,R) \rightarrow \RR$ be a locally-Lipschitz function, and let $x,y \in B(0,R)$ satisfy $|x-y| = 2R$.
Recall that $p = (1 + \alpha) k$. Then,
\begin{equation}  |f(x) - f(y)| \leq  4 \frac{1+\alpha}{\alpha} \cdot k^{\frac{1}{2(1 + \alpha)}}
\cdot  R^{1 - \frac{k}{p}} \left( \int_{B(0,R)} |\nabla f(x)|^p dx \right)^{1/p}.
\label{eq_1420} \end{equation} \label{lem_1547}
\end{lemma}

\begin{proof} We may reduce matters to the case $R=1$ by replacing $f(x)$ by $f(R x)$; note that the right-hand side of (\ref{eq_1420})
is invariant under such replacement.
Thus $x$ is a unit vector, and $y = -x$. Let $Z$ be a random point, distributed uniformly
in the $(k-1)$-dimensional unit ball
$$ B(0,1) \cap x^{\perp} = \{ v \in \RR^k \, ; \, |v| \leq 1, \, v \cdot x = 0 \}, $$
where $v \cdot x$ is the standard scalar product of $x,v \in \RR^k$. Let us write
\begin{align} \label{eq_1448}\EE |f(x) - f(Z)| & \leq \EE |x-Z| \int_0^1 \left|\nabla f( (1 - t) x + t Z )\right| dt
\\ & \leq 2 \EE |\nabla f((1 - T) x + T Z)|  = 2 \int_{B(0,1)} |\nabla f(z)| \rho(z) dz, \nonumber
\end{align}
where $T$ is a random variable uniformly distributed in $[0,1]$, independent of $Z$, and
where
$\rho$ is the probability density of the random variable $(1 - T) x + T Z$. Then,
$$ \rho((1-r)x + r z) = \frac{c_k}{r^{k-1}} $$
when $z \in B(0,1) \cap x^{\perp}, 0 < r < 1$. We may compute $c_k$ as follows:
$$ 1 = c_k \int_0^1 \frac{1}{r^{k-1}} V_{k-1}(r) dr = c_k V_{k-1}(1) = c_k \frac{\pi^{k-1}}{\Gamma \left( \frac{k+1}{2} \right)},
$$
where  $V_{k-1}(r)$ is the $(k-1)$-dimensional volume of $(k-1)$-dimensional Euclidean ball of radius $r$.
Denote $q = p / (p-1)$. Then,
$$ \int_{B(0,1)} \rho^q =  \int_0^1 \left( \frac{c_k}{r^{k-1}} \right)^q  V_{k-1}(r) dr
= \frac{c_k^q V_{k-1}(1)}{(k-1)(1-q) + 1} = \frac{p-1}{p-k} \left(  \frac{\Gamma \left( \frac{k+1}{2} \right)}{\pi^{k-1}} \right)^{q-1},
$$
and hence
\begin{align}  \label{eq_1502} \left( \int_{B(0,1)} \rho^{q} \right)^{1/q} & = \left( \frac{p-1}{p-k} \right)^{1/q} \left(  \frac{\Gamma \left( \frac{k+1}{2} \right)}{\pi^{k-1}} \right)^{1/p}
 \\ & \leq \left( \frac{1+\alpha}{\alpha} \right)^{1/q} \left(  \frac{k^{k/2}}{\pi^{k-1}} \right)^{1/p} \leq  \frac{1+\alpha}{\alpha} \cdot k^{\frac{1}{2(1 + \alpha)}}.  \nonumber \end{align}
Denote $C_{\alpha, k} = \frac{1+\alpha}{\alpha} \cdot k^{\frac{1}{2(1 + \alpha)}}$. From (\ref{eq_1448}), (\ref{eq_1502}) and the H\"older inequality,
\begin{equation}  \EE |f(x) - f(Z)| \leq 2 \left( \int_{B(0,1)} |\nabla f|^p \right)^{\frac{1}{p}} \left( \int_{B(0,1)} \rho^{q} \right)^{\frac{1}{q}}
\leq 2 C_{\alpha, k} \left( \int_{B(0,1)} |\nabla f|^p \right)^{\frac{1}{p}}. \label{eq_1732} \end{equation}
A bound similar to (\ref{eq_1732}) holds also for $\EE |f(y) - f(Z)|$, since $y = -x$. By the triangle inequality,
\begin{equation*} |f(x) - f(y)| \leq \EE |f(y) - f(Z)| + \EE |f(Z) - f(x)| \leq 4 C_{\alpha,k} \left( \int_{B(0,1)} |\nabla f|^p \right)^{1/p}. \tag*{\qedhere} \end{equation*}
\end{proof}

\begin{proof}[Proof of Theorem \ref{thm_138}]
According to Corollary  \ref{lem:choiceofM} we may pick
a coordinate subtorus $M=\mathbb{T}^{k}$
so that
\begin{equation}
\left(\int_{M}|\nabla_{M}f|^{p}\right)^{1/p}\leq
\frac{\alpha}{8(1+\alpha)} \cdot \frac{\eps}{k}
\label{eq_950} \end{equation}
Given any two points ${x},{y}\in M$, let us show that
\begin{equation}
|f({x}) - f({y})| \leq \eps.
\label{eq_1536}
\end{equation}
The distance between ${x}$ and ${y}$ is at most $\sqrt{k}/2$.
Let us construct a curve, in fact a polygonal line, starting at ${x}$ and ending
at ${y}$ which consists of at most $\sqrt{k} + 1$ intervals
of length $1/2$. For instance, we may take all but the last two intervals
to be intervals of length $1/2$ lying on the geodesic between ${x}$ to ${y}$.
The last two intervals need to connect two points whose distance is at most $1/2$, and this is easy to do
by drawing an isosceles triangle whose base is the segment between these two points. \\

Let $[x_j,x_{j+1}]$ be any of the intervals appearing in the polygonal line constructed above.
Let $B \subset \TT^k = M $ be a geodesic ball of radius $R = 1/4$ centered at the midpoint of $[x_j,x_{j+1}]$.
This geodesic ball on the torus is isometric to a Euclidean ball of radius $R = 1/4$ in $\RR^k$. Lemma
\ref{lem_1547} applies, and implies that
$$ |f(x_j) - f(x_{j+1})| \leq 4 \frac{1+\alpha}{\alpha} \cdot k^{\frac{1}{2(1 + \alpha)}} \left( \int_B |\nabla_M f|^p \right)^{\frac{1}{p}} \leq
4 \frac{1+\alpha}{\alpha} \cdot \sqrt{k} \left( \int_M |\nabla_M f|^p \right)^{\frac{1}{p}}. $$
Since the  number of intervals in the polygonal line are at most  $\sqrt{k} + 1 \leq 2 \sqrt{k}$, then
$$ |f(x) - f(y)| \leq \sum_j |f(x_j) - f(x_{j+1})| \leq 8  \frac{1+\alpha}{\alpha} \cdot k    \left( \int_M |\nabla_M f|^p \right)^{1/p}
\leq \eps, $$
where we used (\ref{eq_950}) in the last passage.
The points $x,y \in M$ were arbitrary, hence $\Osc(f ; M) \leq \eps$.
\end{proof}

{\bf Remarks.}
\begin{enumerate}
\item
It is evident from the proof of Theorem \ref{thm_138}
that the subtorus $M$ is chosen randomly and uniformly
over the collection of all $k$-dimensional coordinate subtori. It is easy
to  obtain that with probability at least $9/10$,
we have that $Osc(M ;f) \leq \eps$.
\item The assumption that $f$ is locally-Lipschitz in Theorem \ref{thm_138} is only used
to justify the use of the fundamental theorem of calculus in (\ref{eq_1448}).
It is possible to significantly weaken this assumption; It suffices to know
that $f$ admits weak derivatives $\partial_1 f, \ldots, \partial_n f$ and that (\ref{eq_1714}) holds true,
see \cite[Chapter 4]{EG} for more information.
 \\

It is a bit surprising that the conclusion of the theorem holds also for non-continuous, unbounded functions,
with many singular points, as long as (\ref{eq_1714}) is satisfied in the sense of weak derivatives.
The singularities are necessarily of a rather mild type,
and a variant of our proof yields a subtorus $M$ on which the function $f$ is necessarily continuous with $\Osc(f; M) \leq \eps$.

\item
Another possible approach to the problem would be along the lines
of the proof of the classical concentration theorems - namely, finding
an $\eps$-net of points in a subtorus, where all the coordinate
partial derivatives of the function are small. However, this approach
requires some additional a-priori data about the function, such as
a uniform bound on the Hessian.

\item We do not know whether the dependence on the dimension in Theorem
\ref{thm_1219} is optimal. Better estimates may be obtained if the subtorus $M \subset \TT^n$
is allowed to be an arbitrary $k$-dimensional rational subtorus, which is not necessarily a coordinate subtorus.

\end{enumerate}

\end{document}